\documentclass{amsart}

\usepackage{pstricks,multido}
\usepackage{hyperref}
\usepackage{amsthm}

\newtheorem{thm}{Theorem}[section]
\newtheorem{lem}[thm]{Lemma}
\newtheorem{prop}[thm]{Proposition}
\newtheorem{cor}[thm]{Corollary}
\theoremstyle{definition}
\newtheorem{remark}{Remark}[section]
\newtheorem{example}{Example}[section]
\newtheorem{problem}{Problem}[section]

\newcommand\PT{\mathcal{PT}}
\newcommand\PTB{\mathcal{PT}_\mathrm{B}}
\newcommand\topone{\operatorname{topone}}
 \newcommand\toponez{\operatorname{top}_{0,1}}
\renewcommand\neg{\operatorname{neg}}
\newcommand\diag{\operatorname{diag}}
\newcommand\urc{\operatorname{urc}}
\newcommand\sgn{\operatorname{sign}}
\newcommand\urr{\operatorname{urr}}
\newcommand\RLmin{\operatorname{RLmin}}
\newcommand\RLmax{\operatorname{RLmax}}
\newcommand\zp{\zeta^{\rm perm}}
\newcommand\za{\zeta^{\rm alt}}
\newcommand\zaB{\zeta^{\rm alt}_\mathrm{B}}

\newcommand\CN{\Phi}
\newcommand\CNB{\Phi_\mathrm{B}}

\newcommand\invCNB{\Psi_\mathrm{B}}

\psset{unit=15pt}
\newcount\ax \newcount\ay
\newcount\bx \newcount\by
\newcount\cx \newcount\cy
\newcount\dx \newcount\dy
\def\cell(#1,#2)[#3]{
\ax=#2 \ay=#1
\multiply\ay by-1
\bx=\ax \by=\ay
\cx=\ax \cy=\ay
\dx=\ax \dy=\ay
\advance\bx by-1
\advance\dy by1
\advance\cx by-1
\advance\cy by1
\psline (\dx,\dy)(\ax,\ay)(\bx,\by)
\rput(\number\cx.5,
\ifnum\cy=0 -0.5\else\number\cy.5\fi){#3}
}

\def\colnum[#1,#2]{\rput[b](#1.5,.2){$#2$}}
\def\rownum[#1,#2]{\rput[r](-.2,-#1.5){$#2$}}

\def\UP{\rput(0,-.3){\psline[linewidth=1.5pt,arrowsize=.5, arrowlength=.6]{->}(0, .6)}}
\def\LT{\rput(.3,0){\psline[linewidth=1.5pt,arrowsize=.5, arrowlength=.6]{->}(-.6, 0)}}
\def\DT{\pscircle*{.2}}

\begin{document}

\title{Combinatorics on permutation tableaux of type A and type B}

\author{Sylvie Corteel and Jang Soo Kim}

\date{\today}

\begin{abstract}
  We give two bijective proofs of a result of Corteel and Nadeau. We find a
  generating function related to unrestricted columns of permutation
  tableaux. As a consequence, we obtain a sign-imbalance formula for permutation
  tableaux. We extend the first bijection of Corteel and Nadeau between
  permutations and permutation tableaux to type B objects. Using this type B
  bijection, we generalize a result of Lam and Williams.  We prove that the
  bijection of Corteel and Nadeau and our type B bijection can be expressed as
  zigzag maps on the alternative representation.
\end{abstract}

\maketitle

\section{Introduction}

A permutation tableau is a relatively new combinatorial object coming from
Postnikov's study of totally nonnegative Grassmanian \cite{Postnikov}.  As its
name suggests, permutation tableaux are in bijection with permutations.
Surprisingly, there is also a connection between permutation tableaux and a
statistical physics model called partially asymmetric exclusion process (PASEP),
see \cite{Corteel2007, Corteel2007a, Corteel2007b}.  Recently, several papers on
the combinatorics of permutation tableaux have been published, see
\cite{Burstein2007, Corteel2009, Nadeau, Steingrimsson2007, Viennot}.  In this
paper we study in more detail the combinatorics of permutation tableaux of type
A and type B.

A \emph{permutation tableau} (of type A) is a Ferrers diagram with possibly
empty rows together with a $0,1$-filling of the cells satisfying the following
conditions:
\begin{enumerate}
\item each column has at least one $1$,
\item there is no $0$ which has a $1$ above it in the same column and a $1$ to
  the left of it in the same row.
\end{enumerate}
See Figure~\ref{fig:permalt} for an example of a permutation tableau.  The
\emph{length} of a permutation tableau is defined to be the number of rows plus
the number of columns.  We denote by $\PT(n)$ the set of permutation tableaux of
length $n$.  A $0$ of a permutation tableau is called \emph{restricted} if there
is a $1$ above it in the same column. If a row has no restricted $0$, then it is
called \emph{unrestricted}.

Let $(x)_n$ denote the rising factorial, i.e.~$(x)_0=1$ and
  $(x)_n=x(x+1)\cdots(x+n-1)$ for $n\geq1$. Note that $(2)_{n-1}=(1)_n=n!$ for
  $n\geq1$.

  Using recurrence relations, Corteel and Nadeau \cite[Proposition
  1]{Corteel2009} proved the following:
\begin{equation}
  \label{eq:8}
  \sum_{T\in\PT(n)} x^{\urr(T)-1} y^{\topone(T)} = (x+y)_{n-1},
\end{equation}
where $\urr(T)$ is the number of unrestricted rows of $T$ and $\topone(T)$ is
the number of $1$'s in the first row. In this paper we give two bijective proofs
of \eqref{eq:8}.

Similarly to the definition of unrestricted row, a column of a permutation
tableau is called \emph{unrestricted} if it does not have a $0$ which has a $1$
to the left of it in the same row. For $T\in\PT(n)$, let $\urc(T)$ denote the
number of unrestricted columns of $T$. Then we prove the following.

\begin{thm}\label{thm:gf}
We have
$$\sum_{n\geq0} \sum_{T\in \PT(n)} t^{\urc(T)} x^n = \frac{1+E_t(x)}{1+(t-1)x E_t(x)},$$  
where 
$$E_t(x)=\sum_{n\geq1} n (t)_{n-1} x^n.$$
\end{thm}

There are two interesting special cases $t=2$ and $t=-1$ of Theorem~\ref{thm:gf}
as follows. 

\begin{cor}\label{cor:connected}
  Let $f(n)$ denote the number of connected permutations of $\{1,2,\dots,n\}$,
  see Section~\ref{sec:unrestricted} for the definition. Then
$$\sum_{T\in \PT(n)} 2^{\urc(T)} =f(n+1).$$
\end{cor}

In Section~\ref{sec:unrestricted} we also give a combinatorial proof of
Corollary~\ref{cor:connected}.

We define the \emph{sign} of $T\in\PT(n)$ to be $\sgn(T)=(-1)^{\urc(T)}$.
When $t=-1$ in Theorem~\ref{thm:gf}, we get the following sign-imbalance formula
for permutation tableaux.

\begin{cor}\label{cor:sign}
  We have
$$\sum_{T\in\PT(n)}\sgn(T) =\frac{(1+i)^n + (1-i)^n}{2}.$$
\end{cor}

We also consider the set $\PTB(n)$ of type B permutation tableaux of length
$n$ defined by Lam and Williams \cite{Lam2008a}. Using a similar argument in
\cite{Corteel2009}, Lam and Williams \cite[Proposition 11.4]{Lam2008a} showed
the following (see Section~\ref{sec:typeb} for the definition of $\urr(T)$ and
$\diag(T)$):
\begin{equation}
  \label{eq:9}
  \sum_{T\in\PTB(n)} x^{\urr(T)-1} z^{\diag(T)} = (1+z)^n (x+1)_{n-1}.
\end{equation}

In this paper we find a bijection $\CNB$ between the type B permutation
tableaux and the type B permutations, which extends the first bijection, say
$\CN$, in \cite{Corteel2009}. Using this bijection we prove the following
generalization of \eqref{eq:9}. See Section~\ref{sec:typeb} for the definition
of $\urr(T),\diag(T)$ and $\toponez(T)$.
\begin{thm}\label{thm:typeb}
We have
$$ \sum_{T\in\PTB(n)} x^{\urr(T)-1} y^{\toponez(T)}z^{\diag(T)} = (1+z)^n (x+y)_{n-1}.$$
\end{thm}

There is another representation of a permutation tableau which we call the
\emph{alternative representation}. It was first introduced by Viennot
\cite{Viennota} and systematically studied by Nadeau \cite{Nadeau}.  Nadeau
\cite{Nadeau} found a seemingly different bijection between the permutation
tableaux and the permutations using alternative forests. He also showed that
despite their different descriptions his bijection is, in fact, the same as the
bijection $\CN$ of Corteel and Nadeau \cite{Corteel2009}. We show that this
bijection can also be described as a zigzag map on the alternative
representation similar to the zigzag map on permutation tableaux in
\cite{Steingrimsson2007}. More generally, we show that our map $\CNB$ can be
described as a zigzag on the alternative representation for type B permutation
tableaux.

The rest of this paper is organized as follows. In Section~\ref{sec:bijective}
we recall the bijection $\CN$ and some properties of it. Then we give two
bijective proofs of \eqref{eq:8}. In Section~\ref{sec:unrestricted} we prove
Theorem~\ref{thm:gf} and Corollaries~\ref{cor:connected} and \ref{cor:sign}. In
Section~\ref{sec:typeb} we define the bijection $\CNB$ and prove
Theorem~\ref{thm:typeb}. In Section~\ref{sec:zigzag} we describe $\CN$ and
$\CNB$ in terms of zigzag maps on the alternative representation. In
Section~\ref{sec:further-study} we suggest some open problems.

\section{Two bijective proofs of Corteel and Nadeau's
  theorem}\label{sec:bijective}

A \emph{Ferrers diagram} is a left-justified arrangement of square cells with
possibly empty rows and columns. The \emph{length} of a Ferrers diagram is the
sum of the number of rows and the number of columns. If a Ferrers diagram is of
length $n$, then we label the steps in the south-east border with $1,2,\ldots,n$
from north-east to south-west. We label a row (resp.~column) with $i$ if the row
(resp.~column) contains the south (resp.~west) step labeled with $i$, see
Figure~\ref{fig:ferres}. We denote the row (resp.~column) labeled with $i$ by
Row $i$ (resp.~Column $i$). The \emph{$(i,j)$-entry} is the cell in Row $i$ and
Column $j$.

\begin{figure}
  \centering
  \begin{pspicture}(-1,1)(6,-4) 
\psline(0,-4)(0,0)(6,0)
\rownum[0,3] \rownum[1,5] \rownum[2,8] \rownum[3,10]
\colnum[0,9] \colnum[1,7] \colnum[2,6] \colnum[3,4] \colnum[4,2] \colnum[5,1]
\cell(1,1)[] \cell(1,2)[] \cell(1,3)[] \cell(1,4)[] \cell(2,1)[] \cell(2,2)[] \cell(2,3)[] \cell(3,1)[]     
  \end{pspicture}
\caption{A Ferrers diagram with labeled rows and columns.}
  \label{fig:ferres}
\end{figure}
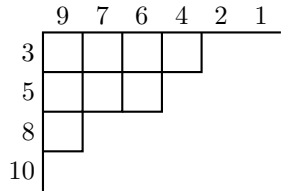

For a permutation tableau $T$, the \emph{alternative representation} of $T$ is
the diagram obtained from $T$ by replacing the topmost $1$'s with $\uparrow$'s,
the rightmost restricted $0$'s with $\leftarrow$'s and removing the remaining
$0$'s and $1$'s, see Figure~\ref{fig:permalt}. Here, by a topmost $1$ we mean a
$1$ which is the topmost $1$ in its column. A rightmost restricted $0$ is
similar.  Note that we can recover the permutation tableau $T$ from its
alternative representation as follows. We fill each empty cell with a $0$ if
there is an arrow pointing to it and with a $1$ otherwise. And then we replace
each $\uparrow$ with a $1$ and each $\leftarrow$ with a $0$.  It is easy to see
that in the alternative representation of $T$ there is no arrow pointing to
another arrow and each column has exactly one $\uparrow$. Conversely, a filling
of a Ferrers diagram with $\uparrow$'s and $\leftarrow$'s satisfies these
conditions if and only if it is the alternative representation of a permutation
tableau.

\begin{remark}
  Our definition of the alternative representation comes from alternative
  tableaux introduced by Viennot \cite{Viennot}. An alternative tableau is
  obtained by deleting the first row of the alternative representation of a
  permutation tableau. Alternative tableaux have the nice symmetric property for
  rows and columns. Rather than using the term alternative tableau, we use the
  term alternative representation to emphasize that we consider it as a
  permutation tableau. See \cite{Nadeau} for more information on alternative
  tableaux.
\end{remark}

\begin{figure}
  \centering
  \begin{pspicture}(-1,1)(6,-7) 
\psline(0,-7)(0,0)(6,0)
\cell(1,1)[0] \cell(1,2)[0] \cell(1,3)[1]
    \cell(1,4)[0] \cell(1,5)[0] \cell(1,6)[1] \cell(2,1)[0] \cell(2,2)[0]
    \cell(2,3)[0] \cell(2,4)[1] \cell(2,5)[1] \cell(2,6)[1] \cell(3,1)[0]
    \cell(3,2)[0] \cell(3,3)[0] \cell(3,4)[0] \cell(3,5)[1] \cell(4,1)[0]
    \cell(4,2)[1] \cell(4,3)[1] \cell(5,1)[1] \cell(6,1)[1] 
    \colnum[0,12] \colnum[1,9] \colnum[2,8] \colnum[3,6] \colnum[4,5]
    \colnum[5,3] \rownum[0,1] \rownum[1,2] \rownum[2,4] \rownum[3,7]
    \rownum[4,10] \rownum[5,11] \rownum[6,13]
\end{pspicture}\hfill
  \begin{pspicture}(-1,1)(6,-7) 
\psline(0,-7)(0,0)(6,0)
\cell(1,1)[] \cell(1,2)[] \cell(1,3)[\UP] \cell(1,4)[] \cell(1,5)[] \cell(1,6)[\UP] \cell(2,1)[] \cell(2,2)[] \cell(2,3)[\LT] \cell(2,4)[\UP] \cell(2,5)[\UP] \cell(2,6)[] \cell(3,1)[] \cell(3,2)[] \cell(3,3)[] \cell(3,4)[\LT] \cell(3,5)[] \cell(4,1)[] \cell(4,2)[\UP] \cell(4,3)[] \cell(5,1)[\UP] \cell(6,1)[] 
   \colnum[0,12] \colnum[1,9] \colnum[2,8] \colnum[3,6] \colnum[4,5]
    \colnum[5,3] \rownum[0,1] \rownum[1,2] \rownum[2,4] \rownum[3,7]
    \rownum[4,10] \rownum[5,11] \rownum[6,13]
\end{pspicture}\hfill
\caption{A permutation tableau $T$ (left) and the alternative representation of
  $T$ (right).}\label{fig:permalt}
\end{figure}

We denote by $S_n$ the set of permutations of $[n]=\{1,2,\ldots,n\}$. 

In \cite{Corteel2009}, Corteel and Nadeau found two bijections between
permutation tableaux and permutations. We recall their first bijection
$\CN:\PT(n)\to S_n$. Here we use the alternative representation.

Given $T\in\PT(n)$, we first set $\pi$ to be the word of labels of the
unrestricted rows of $T$ arranged in increasing order. For each column, say
Column $i$, of $T$, starting from the leftmost column and proceeding to the
right, if Row $j$ contains a $\uparrow$ in Column $i$ and $i_1<\cdots<i_r$ are
the labels of the rows containing a $\leftarrow$ in Column $i$, we add
$i_1,\ldots,i_r,i$ in this (increasing) order before $j$ in $\pi$. Then $\CN(T)$
is the resulting permutation $\pi$.  For example, if $T$ is the permutation
tableau in Figure~\ref{fig:permalt}, then
\begin{equation}
  \label{eq:6}
\CN(T)=4, 6, 5, 2, 8, 3, 1, 9, 7, 12, 10, 11, 13.
\end{equation}
Corteel and Nadeau \cite{Corteel2009} proved that $\CN:\PT(n)\to S_n$ is a
bijection by providing the inverse map.

For now and later use we define the following terminologies.

For a word $w=w_1\ldots w_n$ of distinct integers, $w_i$ is called a
\emph{right-to-left maximum}, or a \emph{RL-maximum}, if $w_i>w_j$ for all
$j\in\{i+1,i+2,\ldots,n\}$. Similarly, $w_i$ is called a \emph{right-to-left
  minimum}, or a \emph{RL-minimum}, if $w_i<w_j$ for all
$j\in\{i+1,i+2,\ldots,n\}$.  Let $w_{i_1},\ldots,w_{i_k}$ be the RL-maxima of
$w$ with $i_1<i_2<\cdots<i_k$. Note that $i_k=n$. By \emph{the cycles of $w$
  according to RL-maxima}, we mean the cycles
\[
(w_1,\ldots,w_{i_1}), (w_{i_1+1},\ldots,w_{i_2}), \ldots, (w_{i_{k-1}+1},\ldots,w_{i_k}).
\]
We also define \emph{the cycles of $w$ according to RL-minima} in the same way.

Corteel and Nadeau \cite[Theorem 1]{Corteel2009} proved the first item and Nadeau
\cite[Theorem 5.4]{Nadeau} proved the second item in the following lemma.

\begin{lem}\label{lem:topone}
For $T\in\PT(n)$ and $\pi\in S_n$ with $\CN(T)=\pi$, we have the following. 
  \begin{enumerate}
  \item The labels of the unrestricted rows of $T$ are exactly the RL-minima of
    $\pi$. 
  \item The labels of the columns containing a $1$ in the first row of $T$ are
    exactly the RL-maxima of the subword of $\pi$ which is to the left of the
    $1$ in $\pi$.
\end{enumerate}
\end{lem}

For example, the labels of the columns containing a $1$ in the first row of $T$ in
Figure~\ref{fig:permalt} are $3$ and $8$ which are the RL-maxima of the subword
$4,6,5,2,8,3$ of $\CN(T)$ in \eqref{eq:6}.

Now we give two bijective proofs of the following theorem which is first proved
by Corteel and Nadeau \cite[Proposition 1]{Corteel2009} using a recurrence
relation. In fact, in the first proof we give a bijection is between objects
counted by the values of the polynomial evaluated at positive integers, whereas
in the second proof we give a bijection between objects counted by the
coefficients of the polynomials. 

\begin{thm}\label{thm:CN}
  We have
\[\sum_{T\in\PT(n)} x^{\urr(T)-1} y^{\topone(T)} = (x+y)_{n-1}.\]
\end{thm}
\begin{proof}[First proof of Theorem~\ref{thm:CN}]
  Since both sides of the equation are polynomials in $x$ and $y$, it is
  sufficient to prove this for all positive integers $x$ and $y$.  Assume $x$
  and $y$ are positive integers and let $N=n+x+y-2$.

  Given $T\in\PT(n)$, we construct $T'\in\PT(N)$ as follows. Firstly, we add
  $y-1$ rows and $x-1$ columns to $T$ by adding $y-1$ south steps at the
  beginning and $x-1$ west steps at the end. Secondly, we fill the first entries
  of the leftmost $x-1$ columns with $\uparrow$'s, see Figure~\ref{fig:cons}.

\begin{figure}
\begin{pspicture}(-10,0)(13,-9.5) 
\rput(-10,-3){
\cell(1,1)[] \cell(1,2)[\UP] \cell(1,3)[\UP] \cell(1,4)[] \cell(1,5)[] \cell(1,6)[\UP] \cell(2,1)[] \cell(2,2)[] \cell(2,3)[\LT] \cell(2,4)[\UP] \cell(2,5)[\UP] \cell(3,1)[\UP] \cell(3,2)[] \cell(3,3)[] \cell(4,1)[\LT] \cell(4,2)[] 
\rput(-4,3){\psline[linewidth=3pt](4,-7)(4,-8)
\pspolygon[linewidth=3pt](4,-3)(4,-7)(6,-7)(6,-6)(7,-6)(7,-5)(9,-5)(9,-4)(10,-4)(10,-3)}}
\rput(2,-8.5){\psscaleboxto(4,.5){\rotateleft{\{}}}\rput(2,-9.2){$x-1$}
\rput(10.5,-1.5){\psscaleboxto(.5,3){\}}}\rput(12,-1.5){$y-1$}
\psframe[fillstyle=solid,fillcolor=yellow](5,0)(6,-4)
\psframe[fillstyle=solid,fillcolor=yellow](6,0)(7,-4)
\psframe[fillstyle=solid,fillcolor=yellow](9,0)(10,-4)
\psframe[fillstyle=solid,fillcolor=green](0,-5)(4,-6)
\psframe[fillstyle=solid,fillcolor=green](0,-7)(4,-8)
\psline(0,-8)(0,0)(10,0)
\psline[linewidth=3pt](4,-7)(4,-8)
\pspolygon[linewidth=3pt](4,-3)(4,-7)(6,-7)(6,-6)(7,-6)(7,-5)(9,-5)(9,-4)(10,-4)(10,-3)
\cell(1,1)[\UP] \cell(1,2)[\UP] \cell(1,3)[\UP] \cell(1,4)[\UP] \cell(1,5)[] \cell(1,6)[] \cell(1,7)[] \cell(1,8)[] \cell(1,9)[] \cell(1,10)[] \cell(2,1)[] \cell(2,2)[] \cell(2,3)[] \cell(2,4)[] \cell(2,5)[] \cell(2,6)[] \cell(2,7)[] \cell(2,8)[] \cell(2,9)[] \cell(2,10)[] \cell(3,1)[] \cell(3,2)[] \cell(3,3)[] \cell(3,4)[] \cell(3,5)[] \cell(3,6)[] \cell(3,7)[] \cell(3,8)[] \cell(3,9)[] \cell(3,10)[] \cell(4,1)[] \cell(4,2)[] \cell(4,3)[] \cell(4,4)[] \cell(4,5)[] \cell(4,6)[\UP] \cell(4,7)[\UP] \cell(4,8)[] \cell(4,9)[] \cell(4,10)[\UP] \cell(5,1)[] \cell(5,2)[] \cell(5,3)[] \cell(5,4)[] \cell(5,5)[] \cell(5,6)[] \cell(5,7)[\LT] \cell(5,8)[\UP] \cell(5,9)[\UP] \cell(6,1)[] \cell(6,2)[] \cell(6,3)[] \cell(6,4)[] \cell(6,5)[\UP] \cell(6,6)[] \cell(6,7)[] \cell(7,1)[] \cell(7,2)[] \cell(7,3)[] \cell(7,4)[] \cell(7,5)[\LT] \cell(7,6)[] \cell(8,1)[] \cell(8,2)[] \cell(8,3)[] \cell(8,4)[] 
\end{pspicture}
 \caption{The construction of $T'$ from $T$.}
    \label{fig:cons}
  \end{figure}

  Let $U$ be the set of labels of the columns of $T$ containing a $\uparrow$ in
  the first row and let $L$ be the set of labels of the unrestricted rows of $T$
  except Row $1$. Now consider all permutation tableaux obtained from $T'$ by
  doing the following: (1) for each $u\in U$, move the $\uparrow$ in the
  $(y,u+y-1)$-entry to the $(i,u+y-1)$-entry for some $1\leq i\leq y$ and (2)
  for each $\ell\in L$, put a $\leftarrow$ in the $(\ell+y-1,j)$-entry for some
  $N-x+2\leq j\leq N$ or do nothing. See Figure~\ref{fig:cons}, where the
  possible cells for $\uparrow$'s and $\leftarrow$'s are colored yellow and
  green respectively. Clearly, there are $x^{\urr(T)-1} y^{\topone(T)}$ such
  permutation tableaux for given $T$. Thus the number of permutation tableaux
  obtained in this way for all $T\in\PT(n)$ equals the left hand side of the
  equation in the theorem. Moreover, it is easy to see that such permutation
  tableaux are exactly those in $\PT(N)$ satisfying the following properties.
  \begin{enumerate}
  \item Each integer $1\leq i\leq y$ is the label of an unrestricted row.
  \item Each integer $N-x+2\leq j\leq N$ is the label of a column containing a
    $\uparrow$ in the first row.
 \end{enumerate}

 By Lemma~\ref{lem:topone}, $Q\in\PT(N)$ satisfies the above two conditions if
 and only if in the permutation $\pi=\CN(Q)\in S_N$ the integers
 $N,N-1,\ldots,N-x+2,1,2,\ldots,y$ are arranged in this order. Since the number
 of such permutations is $\frac{(n+x+y-2)!}{(x+y-1)!}=(x+y)_{n-1}$, we are done.
\end{proof}

\begin{proof}[Second proof of Theorem~\ref{thm:CN}]
  Let $c(n,k)$ denote the number of permutations of $[n]$ with $k$ cycles. The
  following is well known, see \cite[Proposition 1.3.4]{Stanley1997}:
\[
\sum_{k=0}^n c(n,k)x^k = x(x+1)\cdots(x+n-1).
\]
Substituting $n$ and $x$ by $n-1$ and $x+y$ respectively in the above equation,
we obtain
\[ 
\sum_{i,j} c(n-1,i+j)\binom{i+j}{i}x^i y^j = (x+y)_{n-1}.
\] 
Thus it is sufficient to show that the number of $T\in\PT(n)$ with $\urr(T)-1=i$
and $\topone(T)=j$ is equal to $c(n-1,i+j)\binom{i+j}{i}$. For such a
permutation tableau $T$, let $\pi=\CN(T)$ and assume that $\pi$ is decomposed as
$\sigma 1 \tau$. By Lemma~\ref{lem:topone}, the number of RL-minima of $\tau$ is
$i$ and the number of RL-maxima of $\sigma$ is $j$.  Consider the set $C_1$ of
cycles of $\tau$ according to RL-minima and the set $C_2$ of cycles of $\sigma$
according to RL-maxima. Then $C_1$ and $C_2$ have $i$ and $j$ cycles respectively,
and their union forms a permutation of $\{2,3,\ldots,n\}$ with $i+j$ cycles. For
any permutation of $\{2,3,\ldots,n\}$ with $i+j$ cycles, there are
$\binom{i+j}{i}$ ways to distribute them into $C_1$ and $C_2$. Since we can
reconstruct $\tau$ and $\sigma$ from $C_1$ and $C_2$, the number of $\pi=\CN(T)$
is equal to $c(n-1,i+j)\binom{i+j}{i}$.
\end{proof}

\section{Unrestricted columns and sign-imbalance}\label{sec:unrestricted}

Recall that $\urc(T)$ denotes the number of unrestricted columns of $T$.  We define
\[
P_t(x)=\sum_{n\geq0}  \sum_{T\in \PT(n)} t^{\urc(T)} x^n.
\]
Then Theorem~\ref{thm:gf} is the same as the following.

\begin{thm}\label{thm:ptx}
We have
$$P_t(x) = \frac{1+E_t(x)}{1+(t-1)x E_t(x)},$$  
where 
$$E_t(x)=\sum_{n\geq1} n (t)_{n-1} x^n.$$
\end{thm}
\begin{proof}
We will prove the following equivalent formula:
\begin{equation}
  \label{eq:2}
1+ E_t(x) = P_t(x) + (t-1)x E_t(x) P_t(x).  
\end{equation}

Since the coefficients of $x^n$ in both sides are polynomials in $t$, it is
sufficient to prove \eqref{eq:2} for all positive integers $t$.  Assume that $t$
is a positive integer.

  We use a similar argument as in the proof of Theorem~\ref{thm:CN}. However, we
  will use the permutation tableau itself instead of the alternative
  representation.

  Given $T\in \PT(n)$, let $T'$ be the permutation tableau in $\PT(n+t-1)$
  obtained from $T$ by adding $t-1$ south steps at the beginning, where the
  first $t-1$ rows (the added rows) have $0$'s only.  Now consider all
  permutation tableaux obtained from $T'$ by replacing the first $t-1$ $0$'s
  with $i$ $0$'s and $t-1-i$ $1$'s for some $0\leq i\leq t-1$ in each column if
  it was an unrestricted column in $T$. Note that this replacement never
  violates the condition for a permutation tableau because we add 1's only on
  top of an unrestricted column. Clearly, there are $t^{\urc(T)}$ such
  permutation tableaux for each $T\in\PT(n)$.  On the other hand, it is not
  difficult to see that these permutation tableaux are exactly those in
  $\PT(n+t-1)$ satisfying the following conditions.
\begin{enumerate}
\item Each integer $1\leq i\leq t-1$ is the label of an unrestricted row.
\item The integer $t$ is the label of a row.
\item When we remove the first $t-1$ rows, we get a permutation tableau. 
\end{enumerate}
Thus $P_t(x)$ is the generating function for the number of $P\in \PT(n+t-1)$
satisfying the above three conditions.

By Lemma~\ref{lem:topone}, $P\in\PT(n+t-1)$ satisfies the conditions (1) and (2)
if and only if $1,2,\ldots,t-1$ are RL-minima of $\pi$ and $t$ is an ascent of
$\pi$ for $\pi=\CN(P)\in S_{n+t-1}$. Equivalently, $1,2,\ldots,t-1$ are arranged
in this order and $t$ is not immediately followed by any of $1,2,\ldots,t-1$ in
$\pi$. The number of such $\pi$ is equal to
\begin{equation}
  \label{eq:3}
\frac{(n+t-1)!}{(t-1)!} - (t-1) \cdot \frac{(n+t-2)!}{(t-1)!}=
n\cdot t(t+1)\cdots(t+n-2) = n(t)_{n-1}.
\end{equation}

Thus the left hand side of \eqref{eq:2} is the generating function for the
number of $P\in \PT(n+t-1)$ satisfying the conditions (1) and (2).  Since
$P_t(x)$ is the generating function for the number of $P\in \PT(n+t-1)$
satisfying the conditions (1), (2), and (3), it is enough to show that $(t-1)x
E_t(x) P_t(x)$ is the generating function for the number of $P\in \PT(n+t-1)$
satisfying the conditions (1) and (2), but not (3).

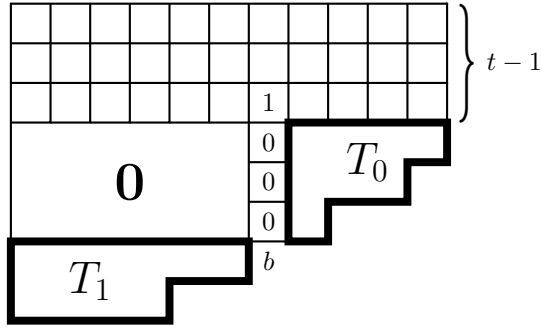
\begin{figure}
  \centering
\begin{pspicture}(0,0)(11,-8)
\psline{C-C}(0,-8)(0,0)(11,0)
\psline(6,-3)(6,-6)(0,-6)
\rput(11.5,-1.5){\psscaleboxto(.5,3){\}}}
\rput(12.7,-1.5){$t-1$}
\cell(1,1)[] \cell(1,2)[] \cell(1,3)[] \cell(1,4)[] \cell(1,5)[] \cell(1,6)[]
\cell(1,7)[] \cell(1,8)[] \cell(1,9)[] \cell(1,10)[] \cell(1,11)[] \cell(2,1)[]
\cell(2,2)[] \cell(2,3)[] \cell(2,4)[] \cell(2,5)[] \cell(2,6)[] \cell(2,7)[]
\cell(2,8)[] \cell(2,9)[] \cell(2,10)[] \cell(2,11)[] \cell(3,1)[] \cell(3,2)[]
\cell(3,3)[] \cell(3,4)[] \cell(3,5)[] \cell(3,6)[] \cell(3,7)[1] \cell(3,8)[]
\cell(3,9)[] \cell(3,10)[] \cell(3,11)[]
\cell(4,7)[0] 
\cell(5,7)[0] 
\cell(6,7)[0] 
\rput(3,-4.5){\bf \Huge 0}
\pspolygon[linewidth=3pt](11,-3)(11,-4)(10,-4)(10,-5)(8,-5)(8,-6)(7,-6)(7,-3)
\pspolygon[linewidth=3pt](0,-6)(6,-6)(6,-7)(4,-7)(4,-8)(3,-8)(0,-8)
\rput(2,-7){\bf \huge $T_1$}
\rput(9,-4){\bf \huge $T_0$}
\rput(6.5,-6.5){$b$}
\end{pspicture}
 \caption{The decomposition of $Q$.}
  \label{fig:decomp}
\end{figure}

Now consider $Q\in\PT(n+t-1)$ satisfying the conditions (1) and (2), but not
(3). We will decompose $Q$ as shown in Figure~\ref{fig:decomp}.  Let $Q'$ be the
tableau obtained from $Q$ by removing the first $t-1$ rows. Then $Q'$ must have
at least one zero column. Let $b$ be the label of the leftmost zero column of
$Q'$.  Since $Q$ is a permutation tableau, in $Q$, Row $t-1$ must have a 1 in
Column $b+t-1$.  Moreover, by the condition for a permutation tableau, all cells
to the left of the cells in zero columns of $Q'$ have $0$'s only. Then in
Figure~\ref{fig:decomp}, $T_1$ together with the part of the first $t-1$ rows of
$Q$ to the left of Column $b+t-1$ is a permutation tableau satisfying the
conditions (1), (2), and (3), and $T_0$ together with the part of the first
$t-1$ rows of $Q$ to the right of Column $b+t-1$ is a permutation tableau
satisfying the conditions (1) and (2). Note that the length of $T_0$ is at least
$1$ because the first step (a south step) of $T_0$ exists, and the length of
$T_1$ may be 0. Since Column $b+t-1$ contributes the factor $(t-1)x$, we obtain
that $(t-1)x E_t(x) P_t(x)$ is the desired generating function.
\end{proof}

For the rest of this section we study the special cases $t=2$ and $t=-1$ of
Theorem~\ref{thm:ptx}.

\subsection{The case $t=2$ : connected permutations}

When $t=2$ in Theorem~\ref{thm:ptx}, we get the following.

\begin{cor}\label{cor:p2x}
We have
\[
P_2(x) =\frac{1}{x} \left(1-\frac{1}{\sum_{n\geq0}n!x^n}\right).
\]
\end{cor}
\begin{proof}
Since
\[E_1(x) = \sum_{n\geq1} n! x^n, \qquad 
E_2(x) = \sum_{n\geq1} n\cdot n! x^n = \frac{1-x}{x} E_1(x) -1,\]
by Theorem~\ref{thm:ptx}, we get
\begin{align*}
P_2(x) &= \frac{1+E_2(x)}{1+x\cdot E_2(x)} 
= \frac{\frac{1-x}{x} E_1(x)}{1+(1-x)E_1(x)-x}\\
&=\frac{1}{x} \cdot \frac{E_1(x)}{1+E_1(x)}
=\frac{1}{x} \left(1-\frac{1}{1+E_1(x)}\right)\\
&=\frac{1}{x} \left(1-\frac{1}{\sum_{n\geq0}n!x^n}\right).
\end{align*}
\end{proof}

A \emph{connected permutation} is a permutation $\pi=\pi_1\cdots\pi_n\in S_n$
such that $\pi$ does not have a proper prefix which is a permutation, in other
words, there is no integer $k<n$ satisfying $\pi_1\cdots\pi_k\in S_k$.  Let
$f(n)$ denote the number of connected permutations of $[n]$. It is known that
$\sum_{n\geq0} f(n) x^n = 1-\left(\sum_{n\geq0}n!x^n\right)^{-1}$, see
\cite{Stanley2005a}. Thus, by Corollary~\ref{cor:p2x}, we get the following.

\begin{cor}\label{thm:connected}
We have
\[
\sum_{T\in\PT(n)} 2^{\urc(T)} = f(n+1).
\]  
\end{cor}

Now we prove Corollary~\ref{thm:connected} combinatorially. To do this we
introduce the following terminology.

A \emph{shift-connected permutation} is a permutation $\pi=\pi_1\cdots\pi_n\in
S_n$ satisfying the following: if $\pi_j = 1$, there is no integer $i<j$ such
that $\pi_i\pi_{i+1}\cdots\pi_{j}\in S_{j-i+1}$. Note that if $\pi_n=1$ then
$\pi$ is not a shift-connected permutation because $\pi_1\cdots\pi_n\in S_n$.

\begin{prop}\label{thm:7}
  The number of connected permutations of $[n]$ is equal to the number of
  shift-connected permutations of $[n]$.
\end{prop}
\begin{proof}
\def\CP{\mathrm{CP}}
\def\SCP{\mathrm{SCP}}
  Let $\CP(n)$ (resp.~$\SCP(n)$) denote the set of connected (resp.~shift-connected)
  permutations of $[n]$. It is sufficient to find a bijection between
  $S_n\setminus \CP(n)$ and  $S_n\setminus \SCP(n)$.

  Given $\pi=\pi_1\cdots\pi_n\in S_n\setminus \CP(n)$, we define $\pi'\in
  S_n\setminus \SCP(n)$ as follows. Since $\pi$ is not a connected permutation,
  we can take the smallest integer $k<n$ such that $\sigma=\pi_1\cdots\pi_k \in
  S_k$.  We decompose $\pi$ as $\pi= \sigma \tau (k+1) \rho$. Define $\pi' =
  \tau \sigma^+ 1 \rho$ where $\sigma^+ = \pi^+_k \pi^+_{k-1} \cdots \pi^+_1$
  with $\pi^+_i = \pi_i +1$ for $i\in[k]$. It is easy to see that
  $\pi\mapsto\pi'$ is a bijection between $S_n\setminus \CP(n)$ and $S_n\setminus
  \SCP(n)$.
\end{proof}

The following proposition gives us a relation between permutation tableaux and
shift-connected permutations.

\begin{prop}\label{thm:8}
  Let $T\in\PT(n)$ and $\pi=\CN(T)$. Then $T$ has a column
  which has a $1$ only in the first row if and only if $\pi$ is not a
  shift-connected permutation.
\end{prop}
\begin{proof}
  Assume that $T$ has a column which has a $1$ only in the first
  row. Then, in the alternative representation, $T$ has a column, say Column
  $d$, which has a $\uparrow$ in the first row and each cell except the topmost
  cell in the column has a $\leftarrow$ in the cell or pointing to it. We will
  refer this condition as $(*)$. 

  Recall the bijection $\CN$. To construct $\pi=\CN(T)$ we first let $\pi$ be
  the word consisting of the labels of the unrestricted rows of $T$ arranged in
  increasing order. Then starting from the leftmost column we add some integers
  to $\pi$. Assume that we have proceeded just before Column $d$.  We can
  decompose $\pi$ as $\sigma 1 \tau$ at this stage. By the condition $(*)$,
  neither $\sigma$ nor $\tau$ has an integer $i$ with $2\leq i\leq d$.  When we
  continue the construction of $\CN$, by the condition $(*)$ again, no integer
  $i$ with $2\leq i\leq d$ is added before the last integer of $\sigma$ or after
  $1$. Thus at the end we must have $\CN(T)=\pi = \sigma \rho 1 \tau$, which
  implies that $\rho1 \in S_d$. Thus $\pi$ is not a shift-connected permutation.

  Conversely, assume that $\pi$ is not a shift-connected permutation. Then we
  can decompose $\pi$ as $\sigma \rho 1 \tau$ such that $\rho 1\in S_d$ for some
  $d$. We claim that Column $d$ in $T$ has the condition $(*)$. Since $d$ is the
  largest integer in $\rho$, it is a RL-maximum in $\sigma\rho$. Thus Column $d$
  in $T$ has a $\uparrow$ in the first row by Lemma~\ref{lem:topone}. If there
  is an unrestricted row with label $1<i<d$, then $i$ is a RL-minimum of $\tau$
  contradicting to $\rho 1\in S_d$. Hence, it remains to show that there is no
  row with label $1<i<d$ having a $\leftarrow$ in a column with label
  $d'>d$. Suppose there is such a row. We can assume that $i$ is the smallest
  integer satisfying this condition. Let $i'$ be the label or the row containing
  the unique $\uparrow$ in Column $d'$. Since $1\leq i' < i$, we must have
  $i'=1$ because otherwise Row $i'$ satisfies the above condition which is a
  contradiction to the minimality of $i$. Thus Column $d'$ has a $\uparrow$ in
  the first row. By Lemma~\ref{lem:topone}, $d'$ is a RL-maximum of
  $\sigma\tau$. Since Row $i$ has a $\leftarrow$ in Column $d'$, we must have
  $i$ before $d'$ in $\pi$. Since both $d$ and $d'$ are RL-maxima of
  $\sigma\tau$ and $d'>d$, we have $d'$ before $d$. Thus the integers $i,d',d$
  appear in this order in $\pi$. This is a contradiction to $\rho 1\in S_d$.
\end{proof}

Using the argument in the proof of Theorem~\ref{thm:ptx} one can easily see that
$\sum_{T\in\PT(n)} 2^{\urc(T)}$ is the number of $T\in\PT(n+1)$ without a column
containing a $1$ only in the first row. Thus, by Proposition~\ref{thm:8}, we get
the following.

\begin{prop}\label{thm:9}
The number of shift-connected permutations of $[n+1]$ is equal to
\[
\sum_{T\in\PT(n)} 2^{\urc(T)}.
\]
\end{prop}

Combining Propositions~\ref{thm:7} and \ref{thm:9}, we get a
combinatorial proof of Corollary~\ref{thm:connected}. 

\subsection{The case $t=-1$ : sign-imbalance}
 
When $t=-1$ in Theorem~\ref{thm:ptx}, we get the following.

\begin{cor}\label{cor:p-1x}
We have
$$P_{-1}(x) = \frac{1-x}{1-2x+2x^2}.$$
\end{cor}
\begin{proof}
Since $E_{-1}(x)=x-2x^2$, we obtain
$$P_{-1}(x) = \frac{1+E_{-1}(x)}{1-2x\cdot E_{-1}(x)}
= \frac{1+x-2x^2}{1-2x^2+4x^3} = \frac{1-x}{1-2x+2x^2}.$$
\end{proof}

Recall that $\sgn(T)=(-1)^{\urc(T)}$ for $T\in\PT(n)$.  Since
\[P_{-1}(x) = \frac{1-x}{1-2x+2x^2} = \frac12 \cdot \left( 
\frac1{1-(1+i)x} +\frac1{1-(1-i)x} \right),\]
we get the following sign-imbalance formula for permutation tableaux. 

\begin{cor}\label{thm:5}
We have
\begin{equation}
  \label{eq:13}
\sum_{T\in\PT(n)}\sgn(T) =\frac{(1+i)^n + (1-i)^n}{2}.
\end{equation}
Equivalently, if $n=4k+r$ for $0\leq r<4$, then
\[
\sum_{T\in\PT(n)}\sgn(T) = \left\{
  \begin{array}{ll}
(-1)^k \cdot 2^{2k}, & \mbox{if $r=0$ or $r=1$,}\\
0, & \mbox{if $r=2$,}\\
(-1)^{k+1} \cdot 2^{2k+1}, & \mbox{if $r=3$.}
 \end{array} \right.
\]
\end{cor}

\section{Type B permutation tableaux}
\label{sec:typeb}

For a Ferrers diagram $F$ with $k$ columns, the \emph{shifted} Ferrers diagram
of $F$, denoted $\overline F$, is the diagram obtained from $F$ by adding $k$
rows of size $1,2,\ldots,k$ above it in increasing order, see
Figure~\ref{fig:shifted}. The \emph{length} of $\overline F$ is defined to be
the length of $F$.  The \emph{diagonal} of $\overline F$ is the set of rightmost
cells in the added rows. A \emph{diagonal cell} is a cell in the diagonal. We
label the added rows as follows. If the diagonal cell of an added row is in
Column $i$, then the row is labeled with $-i$.

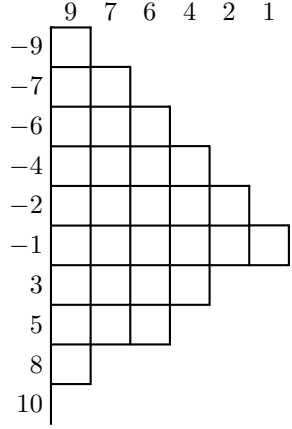
\begin{figure}
  \centering
  \begin{pspicture}(-1,1)(6,-10) 
   \multido{\n=0+1}{6}{\rput(\n,-\n){\psline{C-C}(1,0)}}
    \psline{C-C}(0,0)(0,-10)
    \cell(1,1)[] 
    \cell(2,1)[] \cell(2,2)[] \cell(3,1)[] \cell(3,2)[] 
    \cell(3,3)[] \cell(4,1)[] \cell(4,2)[] \cell(4,3)[]
    \cell(4,4)[] \cell(5,1)[] \cell(5,2)[] \cell(5,3)[]
    \cell(5,4)[] \cell(5,5)[] \cell(6,1)[] \cell(6,2)[] \cell(6,3)[]
    \cell(6,4)[] \cell(6,5)[] \cell(6,6)[] \cell(7,1)[]
    \cell(7,2)[] \cell(7,3)[] \cell(7,4)[] \cell(8,1)[] \cell(8,2)[]
    \cell(8,3)[] \cell(9,1)[] \rownum[0,-9] \rownum[1,-7]
    \rownum[2,-6] \rownum[3,-4] \rownum[4,-2] \rownum[5,-1] \rownum[6,3]
    \rownum[7,5] \rownum[8,8] \rownum[9,10]
    \colnum[0,9] \colnum[1,7] \colnum[2,6] \colnum[3,4] \colnum[4,2] \colnum[5,1] 
  \end{pspicture}
 \caption{The shifted Ferrers diagram $\overline F$ for the Ferrers diagram $F$
    in Figure~\ref{fig:ferres}.}
  \label{fig:shifted}
\end{figure}
 
A \emph{type B permutation tableau} of length $n$ is a shifted Ferrers diagram
$\overline F$ of length $n$ together with a $0,1$-filling of $\overline F$
satisfying the following conditions (see Figure~\ref{fig:alt_typeb}):
\begin{enumerate}
\item each column has at least one $1$,
\item there is no $0$ which has a $1$ above it in the same column and a $1$ to the left of it in the same row,
\item if a $0$ is in a diagonal cell, then there is no $1$ to the left of it in
  the same row.
\end{enumerate}
Note that the condition (3) above means that if we add the copy of the tableau
below the diagonal cells reflected about the diagonal line to get a symmetric
tableau, the condition (2) still holds.

We denote by $\PTB(n)$ the set of type B permutation tableaux of length $n$. 
For $T\in\PTB(n)$, we say that a $0$ is \emph{restricted} if it has a $1$ above
it in the same column or it is in a diagonal cell. If a row of $T$ does not
contain any restricted $0$, it is called an \emph{unrestricted row}.  

\begin{figure}
  \centering
  \begin{pspicture}(-1.5,1)(6,-11)
  \multido{\n=0+1}{6}{\rput(\n,-\n){\psline{C-C}(1,0)}} \cell(1,1)[0]
    \cell(2,1)[0] \cell(2,2)[0] \cell(3,1)[1] \cell(3,2)[0] \cell(3,3)[1]
    \cell(4,1)[0] \cell(4,2)[0] \cell(4,3)[0] \cell(4,4)[0] \cell(5,1)[0]
    \cell(5,2)[0] \cell(5,3)[0] \cell(5,4)[1] \cell(5,5)[1] \cell(6,1)[0]
    \cell(6,2)[0] \cell(6,3)[0] \cell(6,4)[0] \cell(6,5)[0] \cell(6,6)[0]
    \cell(7,1)[0] \cell(7,2)[0] \cell(7,3)[0] \cell(7,4)[0] \cell(7,5)[0]
    \cell(7,6)[1] \cell(8,1)[1] \cell(8,2)[0] \cell(8,3)[1] \cell(8,4)[1]
    \cell(9,1)[0] \cell(9,2)[0] \cell(9,3)[0] \cell(9,4)[1] \cell(10,1)[0]
    \cell(10,2)[1] \cell(10,3)[1] \psline(0,0)(0,-11) \rownum[0,-10]
    \rownum[1,-9] \rownum[2,-8] \rownum[3,-6] \rownum[4,-3] \rownum[5,-2]
    \rownum[6,1] \rownum[7,4] \rownum[8,5] \rownum[9,7] \rownum[10,11]
    \colnum[0,10] \colnum[1,9] \colnum[2,8] \colnum[3,6] \colnum[4,3] \colnum[5,2] 
\end{pspicture} \hfill
\begin{pspicture}(-1.5,1)(6,-11)
 \rput(0,-1){\multido{\n=0+1}{6}{\rput(\n,-\n){\psline(1,0)}}} \cell(2,1)[]
  \cell(3,1)[\UP] \cell(3,2)[] \cell(4,1)[] \cell(4,2)[] \cell(4,3)[]
  \cell(5,1)[] \cell(5,2)[] \cell(5,3)[\LT] \cell(5,4)[\UP] \cell(6,1)[]
  \cell(6,2)[] \cell(6,3)[] \cell(6,4)[] \cell(6,5)[] \cell(7,1)[] \cell(7,2)[]
  \cell(7,3)[] \cell(7,4)[] \cell(7,5)[\LT] \cell(7,6)[\UP] \cell(8,1)[]
  \cell(8,2)[] \cell(8,3)[] \cell(8,4)[] \cell(9,1)[] \cell(9,2)[]
  \cell(9,3)[\LT] \cell(9,4)[] \cell(10,1)[\LT] \cell(10,2)[\UP] \cell(10,3)[]
\rownum[0,-10] \rownum[1,-9] \rownum[2,-8] \rownum[3,-6]
  \rownum[4,-3] \rownum[5,-2] \rownum[6,1] \rownum[7,4] \rownum[8,5]
  \rownum[9,7] \rownum[10,11] \psline{c-c}(0,-11)(0,0)(6,-6)
    \colnum[0,10] \colnum[1,9] \colnum[2,8] \colnum[3,6] \colnum[4,3] \colnum[5,2] 
\end{pspicture} \hfill
\caption{A type B permutation tableau $T$ (left) and the alternative
  representation of $T$ (right).}
  \label{fig:alt_typeb}
\end{figure}
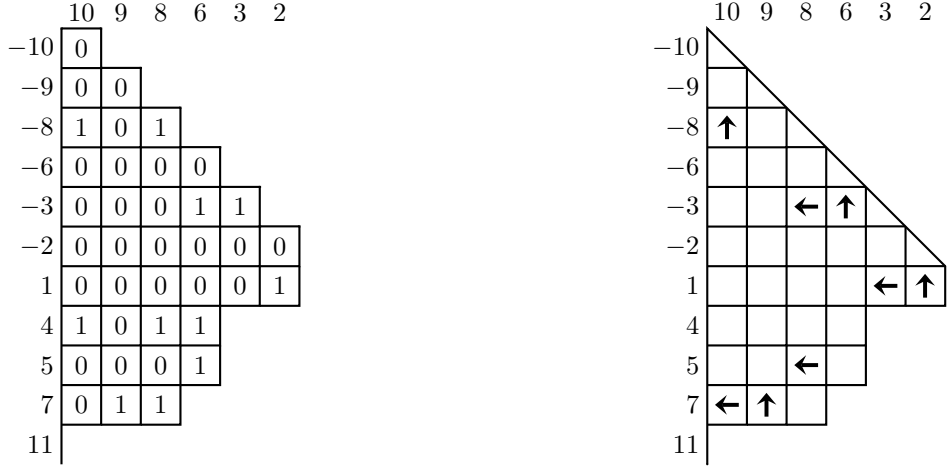
 
The \emph{alternative representation} of $T\in\PTB(n)$ is the diagram obtained
from $T$ as follows. Firstly, we replace the topmost $1$'s with $\uparrow$'s and
the rightmost restricted $0$'s with $\leftarrow$'s and remove the remaining
$0$'s and $1$'s. Secondly, we remove the $\uparrow$'s in the diagonal and cut
off the diagonal cells as shown in Figure~\ref{fig:alt_typeb}.  Note that in the
alternative representation of $T\in\PTB(n)$, there is no arrow pointing to
another arrow. Here we consider the cutting segment as a mirror, that is, an
arrow $\uparrow$ in Column $m$ is pointing to not only the cells above it in
Column $m$ but also the cells in Row $-m$.  Conversely, for any diagram $D$
satisfying this condition, we have $T\in\PTB(n)$ whose alternative
representation is the diagram $D$. In contrast to the alternative representation
of a permutation tableau, we may have a column without a $\uparrow$, which is in
fact hidden in a diagonal cell. We note that the alternative representation of
$T\in\PTB(n)$ is in fact a half of a symmetric alternative tableau in
\cite[Section 3.5]{Nadeau}. 

From now on, we will only consider the alternative representation unless
otherwise stated. Note that Row $i$ is unrestricted if and only if Row $i$ has no
$\leftarrow$ and Column $|i|$, if it exists, has no $\uparrow$.

A \emph{type B permutation} of $[n]$ is a permutation of $[n]$ in which each
integer may be negated. For example, $4,-1,2,5,-3$ is a type B permutation of
$[5]$.  We denote by $B_n$ the set of type B permutations of
$[n]$.

Now we define a map $\CNB:\PTB(n)\to B_n$.  Let $T\in\PTB(n)$. Suppose $T$ has
$k$ columns labeled with $c_1<c_2<\cdots<c_k$. We set $\pi^{(k+1)}$ to be the
word of labels of the unrestricted rows of $T$ arranged in increasing order. For
$i$ from $k$ to $1$, we define $\pi^{(i)}$ to be the word obtained from
$\pi^{(i+1)}$ by adding some integers as follows.  If Column $c_i$ has a
$\uparrow$, let $r_0$ be the label of the row containing the $\uparrow$ and add
$c_i$ before $r_0$ in $\pi^{(i+1)}$.  If Column $c_i$ has no $\uparrow$, then do
nothing (in this case $-c_i$ should be already in $\pi^{(i+1)}$). If
$r_1<\cdots<r_s$ are the labels of the rows containing a $\leftarrow$ in Column
$c_i$, we add $r_1,\ldots,r_s$ in this (increasing) order before $c_i$ or $-c_i$
in $\pi$.  We define $\CNB(T)=\pi^{(1)}$.  For example, if $T$ is the
permutation tableau in Figure~\ref{fig:alt_typeb}, then
\begin{equation}
  \label{eq:7}
\CNB(T)=9,7,10,6,2,1,-3,5,-8,4,11.
\end{equation}

The map $\CNB:\PTB(n)\to B_n$ is a generalization of $\CN:\PT(n)\to S_n$ in the
sense that if $T\in\PTB(n)$ has no diagonal $1$, then we can consider $T$ as an
element in $\PT(n)$ and the map $\CNB$ on such $T$ is identical with $\CN$.

\begin{prop}\label{thm:12}
  The map $\CNB:\PTB(n)\to B_n$ is a bijection.
\end{prop}

We will prove Proposition~\ref{thm:12} by constructing the inverse of $\CNB$.
In order to do this we need the following definitions.  For
$\pi=\pi_1\pi_2\cdots\pi_n\in B_n$ and $i\in[n]$, we say that $\pi_i$ is a
\emph{descent} of $\pi$ if $\pi_i>\pi_{i+1}$, and an \emph{ascent} otherwise,
where $\pi_{n+1}=n+1$. Note that our definition is different from the usual one
which uses the index $i$ as a descent or an ascent instead of the value
$\pi_i$. We also say that $\pi_i$ is a \emph{signed descent} of $\pi$ if
$\pi_i<0$ or $\pi_i>|\pi_{i+1}|$, and a \emph{signed ascent} otherwise. The
\emph{shape} of $\pi$ is the shifted Ferrers diagram such that $|\pi_i|$ is the
label of a column if and only if $\pi_i$ is a signed descent of $\pi$. For
example, the shape of the type B permutation in \eqref{eq:7} is the same as
the shape of the type B permutation tableau in Figure~\ref{fig:alt_typeb}.

Now we define a map $\invCNB:B_n\to\PTB(n)$. Given
$\pi=\pi_1\pi_2\cdots\pi_n\in B_n$, we first set $T$ to be the empty tableau
with the same shape as $\pi$. Let $\pi_{d_1},\pi_{d_2},\ldots,\pi_{d_k}$ be the
signed descents of $\pi$ with $|\pi_{d_1}|<|\pi_{d_2}|<\cdots<|\pi_{d_k}|$. Let
$\pi^{(1)}=\pi$.  For $i$ from $1$ to $k$, we define $\pi^{(i+1)}$ using
$\pi^{(i)}$ as follows. During the process, we will fill the cells in Column
$|\pi_{d_i}|$ of $T$. Suppose $\pi^{(i)}=\sigma_1\sigma_2\cdots\sigma_{\ell}$
and $\pi_{d_i}=\sigma_d$.

\begin{description}
\item[Case 1] If $\sigma_d<0$, let $r$ be the smallest integer with
  $\sigma_d<\sigma_r<\sigma_{r+1}<\cdots<\sigma_{d-1}<|\sigma_{d}|$. We fill the
  $(\sigma_j,|\sigma_d|)$-entry with a $\leftarrow$ for each $r\leq
  j\leq d-1$. Then $\pi^{(i+1)}$ is the word obtained from $\pi^{(i)}$ by
  removing $\sigma_j$ for all $r\leq j\leq d-1$.
\item[Case 2] If $\sigma_d>0$, then we have $\sigma_d>|\sigma_{d+1}|$ because
  $\pi_{d_i}=\sigma_d$ is a signed descent of $\pi^{(i)}$ (this will be proved
  in Lemma~\ref{thm:2}). Let $r$ be the smallest integer with
  $\sigma_{d+1}<\sigma_r<\sigma_{r+1}<\cdots<\sigma_{d-1}<\sigma_{d}$. We fill
  the $(\sigma_{d+1},\sigma_d)$-entry with a $\uparrow$ and the
  $(\sigma_j,\sigma_d)$-entry with a $\leftarrow$ for each $r\leq
  j\leq d-1$. Then $\pi^{(i+1)}$ is the word obtained from $\pi^{(i)}$ by
  removing $\sigma_j$ for all $r\leq j\leq d$.
\end{description}
Repeat this process until we get $\pi^{(k+1)}$ and define $\invCNB(\pi)$ to be
the resulting tableau $T$.  For example, if $\pi$ is the permutation in
\eqref{eq:7}, then $\invCNB(\pi)$ is the permutation tableau in
Figure~\ref{fig:alt_typeb}.  From the construction of $\invCNB$ it is easy to
see that $\pi^{(k+1)}$ consists of the integers which are the labels of the
unrestricted rows of $\invCNB(\pi)$.

\begin{lem}\label{thm:2}
  Following the notations in the definition of $\invCNB$, we have the following.
  \begin{enumerate}
  \item For each $i\in[k]$, $\pi_{d_i}$ is a signed descent in $\pi^{(i)}$.
  \item $\pi^{(k+1)}$ is the word of the RL-minima of $\pi$ arranged in
    increasing order. 
\end{enumerate}
\end{lem}
\begin{proof}
  (1) It is enough to show the following: for $j<i$, when we go from $\pi^{(j)}$
  to $\pi^{(j+1)}$, $\pi_{d_i}$ is not deleted and it remains a signed
  descent. First observe that we have $|\pi_{d_j}|<|\pi_{d_i}|$ by definition.
  Since we only delete integers between $-|\pi_{d_j}|$ and $|\pi_{d_j}|$, $\pi_{d_i}$ is not
  deleted. If $\pi_{d_i}<0$, it is a signed descent. Otherwise, the integer to
  the right of $\pi_{d_i}$ may change. In this case, however, the integer to the
  right of $\pi_{d_i}$ in $\pi^{(j+1)}$ is smaller than $|\pi_{d_j}|$ by the
  construction of $\invCNB$. Thus $\pi_{d_i}$ remains a signed descent in
  $\pi^{(j+1)}$.

(2) Similarly, we can prove the following.
\begin{itemize}
\item For an integer $g$ in $\pi^{(i)}$, $g$ is a descent in $\pi^{(i)}$
  if and only if $g$ is a descent in $\pi$.
\item $\pi^{(k+1)}$ does not have any descent of $\pi$.
\item $\pi^{(k+1)}$ contains all RL-minima of $\pi$.
\end{itemize}
Using the above three statements one can easily derive (2).
\end{proof}

By Lemma~\ref{thm:2}, $\pi^{(k+1)}$ is the RL-minima of $\pi$ arranged in
increasing order and these integers are the labels of the unrestricted rows of
$T=\invCNB(\pi)$. Now it is easy to see that $\CNB(T)=\pi$. Hence we obtain that
the map $\CNB:\PTB(n)\to B_n$ is a bijection with inverse $\CNB^{-1}=\invCNB$.
This completes the proof of Proposition~\ref{thm:12}. 

Now we will find some properties of $\CNB$ analogous to $\CN$. 

Let $T\in\PTB(n)$ and let Row $m$ be the topmost unrestricted row of $T$.  We
denote by $\diag(T)$ the number of columns without a $\uparrow$, by $\urr(T)$
the number of unrestricted rows and by $\toponez(T)$ the number of arrows in Row
$m$ and Column $|m|$. Equivalently, in the original permutation tableau
representation $\diag(T)$ is the number of diagonal cells containing a $1$ and
$\toponez(T)$ is the number of $1$'s in Row $m$ except the $1$ in the diagonal
cell, if it exists, plus the number of rightmost restricted $0$'s in Column
$|m|$. Note that if $m>0$, then we must have $m=1$ and $\toponez(T)$ is simply
the number of $1$'s in Row $1$.

\begin{example}
  Let $T$ be the type B permutation tableau in Figure~\ref{fig:alt_typeb}.
  Then the topmost unrestricted row of $T$ is Row $-8$ and $\diag(T)=2$,
  $\urr(T)=3$ and $\toponez(T)=1+2=3$.
\end{example}

The following is a type B analog of Lemma~\ref{lem:topone}.

\begin{prop}\label{thm:1}
  Let $T\in\PTB(n)$ and $\pi=\CNB(T)$ with $\pi$ decomposed as $\pi=\sigma\tau
  m \rho$, where $\min(\pi)=m$, the last element of $\sigma$ is greater than
  $|m|$ and each element of $\tau$ is smaller than $|m|$. Then we have the
  following.
  \begin{enumerate}
  \item Row $m$ is the topmost unrestricted row of $T$.
 \item Column $t$ has no $\uparrow$ if and only if $-t$ appears in $\pi$.
 \item The labels of the unrestricted rows of $T$ are exactly the RL-minima of
    $\pi$.
 \item The labels of the columns of $T$ with a $\uparrow$ in Row $m$ are
    exactly the RL-maxima of $\sigma$. Here $\sigma$ can be empty.
  \item The labels of the rows of $T$ containing a $\leftarrow$ in Column $|m|$
    are exactly the RL-minima of $\tau$. Here $\tau$ can be empty.
  \item We have $\neg(\pi) = \diag(T)$.
\end{enumerate}
\end{prop}
\begin{proof}
  Let $T$ have $k$ columns labeled with $c_1<c_2<\cdots<c_k$.  Using the
  notations in the definition of $\CNB$, we can observe the following.

  \textbf{Observation}: When we obtain $\pi^{(i)}$ from $\pi^{(i+1)}$, we add an
  increasing sequence of integers $a_1<\cdots<a_s$ before $b$ for some integers
  $a_1,\dots,a_s,b$ satisfying $b<a_s\leq c_i$.

  Then (2) is obvious from the definition of $\CNB$ and (3) follows from the
  fact that we do not create any new RL-minimum when we obtain $\pi^{(i)}$ from
  $\pi^{(i+1)}$ by the above observation. Since the label of the topmost
  unrestricted row is smaller than that of any other unrestricted row, by (3) we obtain
  (1). 

  Now we prove (4) and (5). Let $t$ be the integer with $c_t=|m|$ if $m<0$ and
  $t=k+1$ if $m=1$.  Since $\pi^{(k+1)}$ is the word of the RL-minima of $\pi$
  arranged in increasing order, the first element of $\pi^{(k+1)}$ is $m$. Thus
  for each $i\in[k+1]$ we can decompose $\pi^{(i)}$ as $\pi^{(i)} = \sigma^{(i)}
  \tau^{(i)} m \rho^{(i)}$, where the last element of $\sigma^{(i)}$ is greater
  than $|m|$ and each element of $\tau^{(i)}$ is smaller than $|m|$.

  Using the observation, it is not difficult to see that
  $\RLmax(\sigma^{(i)})=\RLmax(\sigma^{(i+1)})\cup\{|c_i|\}$ if Column $|c_i|$
  has a $\uparrow$ in Row $m$, and $\RLmax(\sigma^{(i)})=\RLmax(\sigma^{(i+1)})$
  otherwise. Thus we get (4). Using the observation again, one can also easily
  see that $\tau^{(k+1)}=\tau^{(k)}=\cdots=\tau^{(t-1)}=\emptyset$, $\tau^{(t)}$
  is the increasing sequence of labels of the rows of $T$ containing a
  $\leftarrow$ in Column $|m|$ and
  $\RLmin(\tau^{(t)})=\RLmin(\tau^{(t-1)})=\cdots=\RLmin(\tau^{(1)})$.  Thus we
  get (5).

  From the construction of $\CNB$, (6) is obvious. 
\end{proof}

Now we are ready to prove the following theorem.

\begin{thm}\label{thm:gf_typeB}
 We have
$$\sum_{T\in\PTB(n)} x^{\urr(T)-1} y^{\toponez(T)}  z^{\diag(T)} = (1+z)^n (x+y)_{n-1}.$$
\end{thm}
Note, if $z=0$, the above theorem reduces to Theorem~\ref{thm:CN} by identifying
a permutation tableau with a type B permutation tableau without any diagonal
$1$. We give two proofs of this theorem. These are type B analogs of the two
proofs of Theorem~\ref{thm:CN}.

\begin{proof}[First proof of Theorem~\ref{thm:gf_typeB}]
 \begin{figure}
    \centering
 \begin{pspicture}(-13,0)(12,-20.5)
\rput(-12,-5){
\multido{\n=1+1}{7}{\rput(\n,-\n){\psline(-1,0)}}
\rput(-5,5){\pspolygon[linewidth=3pt](10,-10)(12,-12)(12,-14)(11,-14)(11,-16)(10,-16)}
\rput(-3,5){\pspolygon[linewidth=3pt](7,-10)(7,-16)(6,-16)(6,-17)(4,-17)(4,-18)(3,-18)(3,-10)}
\rput(-.5,-4.5){$m$}
\rput(4.5,-11.5){$|m|$}
\psline(7,-7)(0,0)(0,-14)
\psline[linewidth=3pt](0,-13)(0,-14)
\rput(0,-1){
\cell(1,1)[] \cell(2,1)[] \cell(2,2)[] \cell(3,1)[] \cell(3,2)[] \cell(3,3)[] \cell(4,1)[] \cell(4,2)[\UP] \cell(4,3)[] \cell(4,4)[\UP] \cell(5,1)[\UP] \cell(5,2)[] \cell(5,3)[] \cell(5,4)[] \cell(5,5)[] \cell(6,1)[] \cell(6,2)[] \cell(6,3)[] \cell(6,4)[] \cell(6,5)[\LT] \cell(6,6)[] \cell(7,1)[] \cell(7,2)[] \cell(7,3)[\UP] \cell(7,4)[] \cell(7,5)[] \cell(7,6)[] \cell(7,7)[] \cell(8,1)[] \cell(8,2)[] \cell(8,3)[] \cell(8,4)[] \cell(8,5)[\LT] \cell(8,6)[] \cell(8,7)[] \cell(9,1)[] \cell(9,2)[] \cell(9,3)[] \cell(9,4)[] \cell(9,5)[] \cell(9,6)[\LT] \cell(10,1)[] \cell(10,2)[] \cell(10,3)[] \cell(10,4)[] \cell(10,5)[\LT] \cell(10,6)[] \cell(11,1)[] \cell(11,2)[\LT] \cell(11,3)[] \cell(12,1)[\LT] }
\psframe[linewidth=2pt, linecolor=red](0,-4)(4,-5)
\psframe[linewidth=2pt, linecolor=red](4,-5)(5,-11)
}

\rput(8,-16.5){\psscaleboxto(2,.5){\rotateleft{\{}}}
\rput(8,-17.2){$y-1$}

\rput(-.5,-8){\psscaleboxto(.5,2){\{}}
\rput[r](-1,-8){$y-1$}

\rput(1.5,-19.5){\psscaleboxto(3,.5){\rotateleft{\{}}}
\rput(1.5,-20.2){$x-1$}

\multido{\n=1+1}{12}{\rput(\n,-\n){\psline(-1,0)}}
\psline (0,-19)(0,0)(12,-12)
\pspolygon[linewidth=3pt](10,-10)(12,-12)(12,-14)(11,-14)(11,-16)(10,-16)
\pspolygon[linewidth=3pt](7,-10)(7,-16)(6,-16)(6,-17)(4,-17)(4,-18)(3,-18)(3,-10)
\psline[linewidth=3pt](3,-18)(3,-19)

\psframe[fillcolor=yellow,fillstyle=solid](7,-11)(10,-12)
\psframe[fillcolor=yellow,fillstyle=solid](7,-13)(10,-14)
\psframe[fillcolor=yellow,fillstyle=solid](7,-15)(10,-16)

\psframe[fillcolor=yellow,fillstyle=solid](7,-10)(6,-7)
\psframe[fillcolor=yellow,fillstyle=solid](5,-10)(4,-7)

\psframe[fillcolor=green,fillstyle=solid](0,-10)(3,-11)
\psframe[fillcolor=green,fillstyle=solid](0,-12)(3,-13)
\psframe[fillcolor=green,fillstyle=solid](0,-18)(3,-19)

\rput(0,-1){
\cell(1,1)[] \cell(2,1)[] \cell(2,2)[] \cell(3,1)[] \cell(3,2)[] \cell(3,3)[]
\cell(4,1)[] \cell(4,2)[] \cell(4,3)[] \cell(4,4)[] \cell(5,1)[] \cell(5,2)[]
\cell(5,3)[] \cell(5,4)[] \cell(5,5)[] \cell(6,1)[] \cell(6,2)[] \cell(6,3)[]
\cell(6,4)[] \cell(6,5)[] \cell(6,6)[] \cell(7,1)[] \cell(7,2)[]
\cell(7,3)[] \cell(7,4)[] \cell(7,5)[] \cell(7,6)[] \cell(7,7)[]
\cell(8,1)[] \cell(8,2)[] \cell(8,3)[] \cell(8,4)[] \cell(8,5)[] \cell(8,6)[]
\cell(8,7)[] \cell(8,8)[] \cell(9,1)[\UP] \cell(9,2)[\UP] \cell(9,3)[\UP] \cell(9,4)[]
\cell(9,5)[\UP] \cell(9,6)[] \cell(9,7)[\UP] \cell(9,8)[] \cell(9,9)[] \cell(10,1)[]
\cell(10,2)[] \cell(10,3)[] \cell(10,4)[\UP] \cell(10,5)[] \cell(10,6)[]
\cell(10,7)[] \cell(10,8)[] \cell(10,9)[] \cell(10,10)[] \cell(11,1)[]
\cell(11,2)[] \cell(11,3)[] \cell(11,4)[] \cell(11,5)[] \cell(11,6)[]
\cell(11,7)[] \cell(11,8)[] \cell(11,9)[] \cell(11,10)[\LT] \cell(11,11)[]
\cell(12,1)[] \cell(12,2)[] \cell(12,3)[] \cell(12,4)[] \cell(12,5)[]
\cell(12,6)[\UP] \cell(12,7)[] \cell(12,8)[] \cell(12,9)[] \cell(12,10)[]
\cell(12,11)[] \cell(12,12)[] \cell(13,1)[] \cell(13,2)[] \cell(13,3)[]
\cell(13,4)[] \cell(13,5)[] \cell(13,6)[] \cell(13,7)[] \cell(13,8)[]
\cell(13,9)[] \cell(13,10)[\LT] \cell(13,11)[] \cell(13,12)[] \cell(14,1)[]
\cell(14,2)[] \cell(14,3)[] \cell(14,4)[] \cell(14,5)[] \cell(14,6)[]
\cell(14,7)[] \cell(14,8)[] \cell(14,9)[] \cell(14,10)[] \cell(14,11)[\LT]
\cell(15,1)[] \cell(15,2)[] \cell(15,3)[] \cell(15,4)[] \cell(15,5)[]
\cell(15,6)[] \cell(15,7)[] \cell(15,8)[] \cell(15,9)[] \cell(15,10)[\LT]
\cell(15,11)[] \cell(16,1)[] \cell(16,2)[] \cell(16,3)[] \cell(16,4)[]
\cell(16,5)[\LT] \cell(16,6)[] \cell(17,1)[] \cell(17,2)[] \cell(17,3)[]
\cell(17,4)[\LT] \cell(18,1)[] \cell(18,2)[] \cell(18,3)[] 
}
\rput(3,-5){\psframe[linewidth=2pt, linecolor=red](0,-4)(4,-5)}
\rput(5,-5){\psframe[linewidth=2pt, linecolor=red](4,-5)(5,-11)}
 \end{pspicture}
 \caption{Expansion}
    \label{fig:expansion}
  \end{figure}
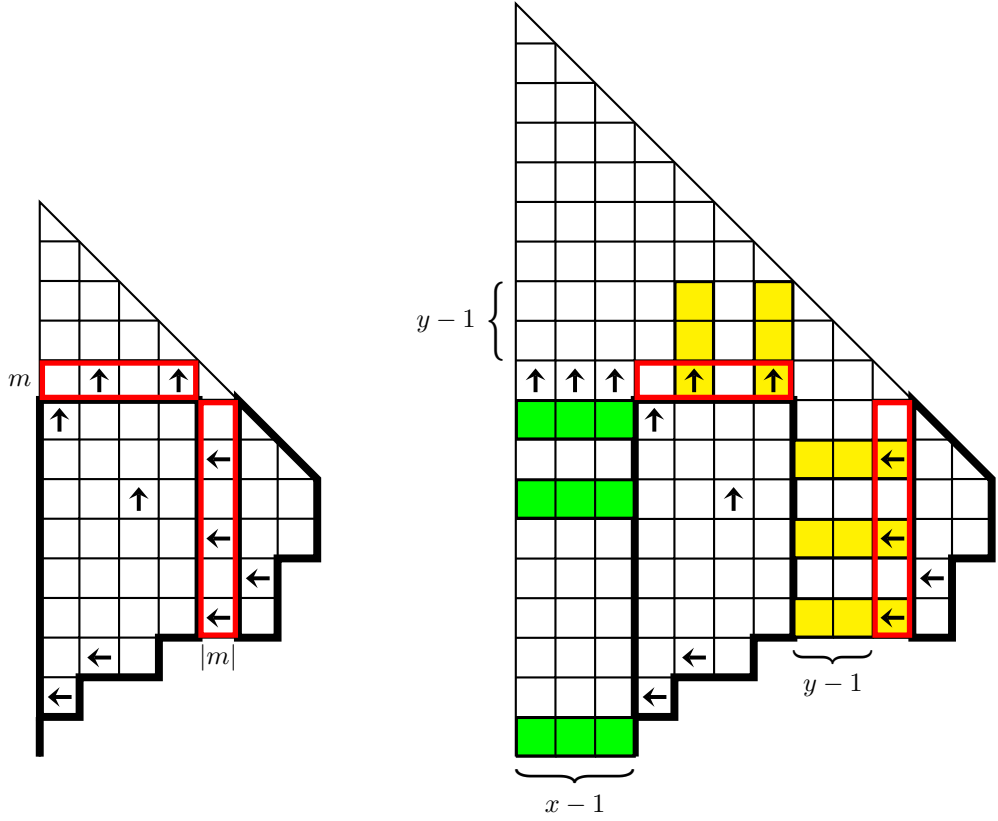

  We consider $x$ and $y$ as positive integers. Let $N=n+x+y-2$.  Note
  $T\in\PTB(n)$ with $\diag(T)=0$ can be considered as an element in
  $\PT(n)$. Thus by Theorem~\ref{thm:CN}, the left hand side of the formula is
  equal to
\begin{equation}
  \label{eq:12}
  (x+y)_{n-1} + 
\sum_{\substack{T\in\PTB(n)\\ \diag(T)\ne0}} x^{\urr(T)-1} y^{\toponez(T)}  z^{\diag(T)}.
\end{equation}

Given $T\in\PTB(n)$ with $\diag(T)\ne0$, we define $T'\in\PTB(N)$ as
follows. Let Row $m$ be the topmost unrestricted row of $T$. Since
$\diag(T)\ne0$ we have $m<0$. Then $T'$ is obtained from $T$ by inserting $y-1$
columns to the left of Column $|m|$, $y-1$ rows above Row $m$ and $x-1$ columns
at the end, where the new $x-1$ cells attached to the left of Row $m$ are fill
with $\uparrow$'s as shown Figure~\ref{fig:expansion}.

As we did in the first proof of Theorem~\ref{thm:CN}, we can move each arrow in
Row $m$ (resp.~Column $|m|$) to one of the $y-1$ cells above it (resp.~to the
left of it), and we can place a $\leftarrow$ in one of the $x-1$ cells to the
left of each unrestricted row, except Row $m$. Let $X(T)$ denote the set of type
B permutation tableaux obtained in this way. Clearly $X(T)$ has
$x^{\urr(T)-1}y^{\toponez(T)}$ elements and for each $Q\in X(T)$ we have
$z^{\diag(Q)} = z^{\diag(T)+y-1}$. Thus \eqref{eq:12} is equal to
\begin{equation}
    \label{eq:11}
  (x+y)_{n-1} + 
   \sum_{\substack{T\in\PTB(n)\\ \diag(T)\ne0}} \sum_{Q\in X(T)} z^{\diag(Q)-y+1}.
\end{equation}

On the other hand, one can check that $Q\in X(T)$ for some $T\in\PTB(n)$ with
$\diag(T)\ne0$ if and only if $Q$ satisfies the following conditions.  Let Row
$M$ be the topmost unrestricted row of $Q$.
\begin{itemize}
\item We have $M\leq -y$ and $|M|\leq N+x-1$.
\item Row $r$ is unrestricted for all $r$ with $M\leq r\leq M+y-1$.
\item Column $c$ has a $\uparrow$ in Row $M$ for all $c$ with $N-x+2\leq
  c\leq N$.
\end{itemize}

Now let $\pi=\CNB(Q)\in B_N$. By Proposition~\ref{thm:1}, $Q$ satisfies the
above conditions if and only if $M=\min(\pi)\leq -y$, $|M|\leq N+x-1$ and the
integers $N,N-1,\ldots,N-x+2, M,M+1,\ldots,M+y-1$ are arranged in this order in
$\pi$. Suppose $\pi\in B_n$ satisfies these conditions. Let $A\subset[N]$ be the
set of integers $a$ such that $\pi$ has $-a$. Since $N,N-1,\ldots,N-x+2$ appear
in $\pi$, we must have $A\subset [n+y-1]$. Since $M\leq -y$ and
$M,M+1,\ldots,M+y-1$ appear in $\pi$, the largest $y$ elements in $A$ are
consecutive. Let $A'$ be the set obtained from $A$ by deleting the largest $y-1$
elements. Then $\emptyset\ne A'\subset[n]$. Note that we can reconstruct $A$
from $A'$. When $A'$ is fixed, the number of such $\pi\in B_N$ is equal to
$\frac{N!}{(x+y-1)!}=(x+y)_{n-1}$.  Since $\diag(Q)=\neg(\pi)$, the number of
negative integers in $\pi$, and $\neg(\pi)=|A|=|A'|+y-1$, we obtain that
\eqref{eq:11} is equal to the following which finishes the proof:
\[
(x+y)_{n-1}+\sum_{\emptyset \ne A'\subset[n]} z^{|A'|}(x+y)_{n-1} =(1+z)^n(x+y)_{n-1}.
\]
\end{proof}

\begin{proof}[Second proof of Theorem~\ref{thm:gf_typeB}]
  Let $T\in\PTB(n)$ and $\pi=\CNB(T)$. We decompose $\pi$ as $\pi=\sigma\tau m
  \rho$, where $m=\min(\pi)$, the last element of $\sigma$ is greater than $|m|$
  and all the entries in $\tau$ are smaller than $|m|$.  Then by
  Proposition~\ref{thm:1}, $\diag(T)=\neg(\pi)$, $\urr(T)-1=\RLmin(\rho)$,
  $\toponez(T)=\RLmax(\sigma)+\RLmin(\tau)$, where for a word $w$ of integers,
  $\RLmin(w)$ (resp.~$\RLmax(w)$) denotes the number of RL-minima
  (resp.~RL-maxima) of $w$.  Thus
  \begin{equation}
    \label{eq:10}
\sum_{T\in\PTB(n)} x^{\urr(T)-1} y^{\toponez(T)} z^{\diag(T)}=
\sum_{\pi\in B_n} z^{\neg(\pi)} x^{\RLmin(\rho)}
y^{\RLmax(\sigma)+\RLmin(\tau)}.
 \end{equation}

 On the other hand, consider the set $C_1$ (resp.~$C_2$ and $C_3$) of cycles of
 $\sigma$ (resp.~$\tau$ and $\rho$) according to RL-maxima (resp.~RL-minima and
 RL-minima). Then 
\[
z^{\neg(\pi)} x^{\RLmin(\rho)} y^{\RLmax(\sigma)+\RLmin(\tau)}
=z^{\neg(\pi)} x^{|C_3|} y^{|C_1|+|C_2|}.
\]
Note that $\pi$ is determined by the ordered triple $(C_1, C_2, C_3)$ and
$m$. Moreover, since each cycle in $C_1$ contains at least one integer greater
than $|m|$ and no cycle in $C_2$ has this property, $\pi$ is determined by the
ordered pair $(C_1\cup C_2, C_3)$ and $m$. Let us fix the set $A\subset[n]$ of
integers $a$ such that $\pi$ has $-a$. Note that $A$ determines the set $S$ of
integers in $\pi$ and $m=\min(S)$. Since $(C_1\cup C_2)\cup C_3$ is the set of
cycles of a permutation on $S\setminus\{m\}$, the right hand side of
\eqref{eq:10} is equal to
\[
\sum_{A\subset [n]} z^{|A|}\sum_{i,j} c(n-1,i+j)\binom{i+j}{i}x^i y^j,
\]
which is equal to $(1+z)^n (x+y)_{n-1}$ as we have shown in the second proof of
Theorem~\ref{thm:CN}.
\end{proof}

\section{Zigzag maps}
\label{sec:zigzag}

Steingr\'imsson and Williams \cite{Steingrimsson2007} defined the \emph{zigzag
  map} $\zp:\PT(n)\to S_n$ on permutation tableaux as follows.  A \emph{zigzag
  path} on a permutation tableau $T\in\PT(n)$ is a path entering from the left
of a row or the top of a column, going to the east or to the south changing the
direction alternatively whenever it meets a $1$ until exiting the tableau, see
Figure~\ref{fig:threetab}. Then $\zp(T)$ is defined to be the permutation
$\pi=\pi_1\cdots\pi_n$ where $\pi_i=j$ if the zigzag path starting from Row $i$
or Column $i$ exits $T$ from Row $j$ or Column $j$. It is easy to see that
$\zp(T)$ is a permutation of $[n]$.

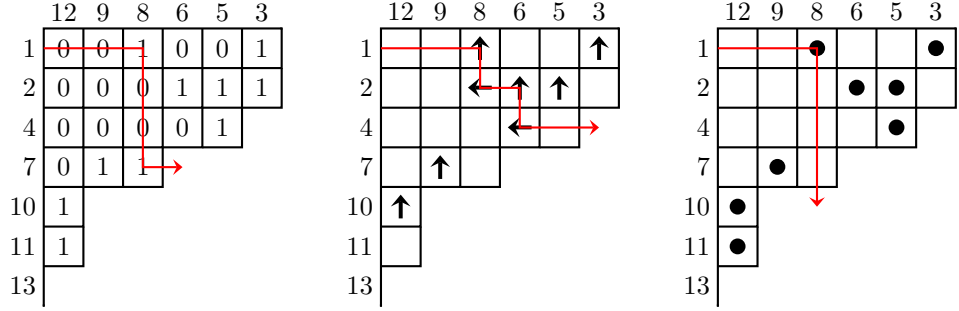
\begin{figure}
  \centering
  \begin{pspicture}(-1,1)(6,-7) 
\psline{C-C}(0,-7)(0,0)(6,0)
\cell(1,1)[0] \cell(1,2)[0] \cell(1,3)[1]
    \cell(1,4)[0] \cell(1,5)[0] \cell(1,6)[1] \cell(2,1)[0] \cell(2,2)[0]
    \cell(2,3)[0] \cell(2,4)[1] \cell(2,5)[1] \cell(2,6)[1] \cell(3,1)[0]
    \cell(3,2)[0] \cell(3,3)[0] \cell(3,4)[0] \cell(3,5)[1] \cell(4,1)[0]
    \cell(4,2)[1] \cell(4,3)[1] \cell(5,1)[1] \cell(6,1)[1] 
    \colnum[0,12] \colnum[1,9] \colnum[2,8] \colnum[3,6] \colnum[4,5]
    \colnum[5,3] \rownum[0,1] \rownum[1,2] \rownum[2,4] \rownum[3,7]
    \rownum[4,10] \rownum[5,11] \rownum[6,13]
\psline[linecolor=red,arrowsize=.4, arrowlength=.6]{->}(0,-.5)(2.5,-.5)(2.5,-3.5)(3.5,-3.5)
\end{pspicture}\hfill
  \begin{pspicture}(-1,1)(6,-7) 
\psline{C-C}(0,-7)(0,0)(6,0)
\cell(1,1)[] \cell(1,2)[] \cell(1,3)[\UP] \cell(1,4)[] \cell(1,5)[] \cell(1,6)[\UP] \cell(2,1)[] \cell(2,2)[] \cell(2,3)[\LT] \cell(2,4)[\UP] \cell(2,5)[\UP] \cell(2,6)[] \cell(3,1)[] \cell(3,2)[] \cell(3,3)[] \cell(3,4)[\LT] \cell(3,5)[] \cell(4,1)[] \cell(4,2)[\UP] \cell(4,3)[] \cell(5,1)[\UP] \cell(6,1)[] 
   \colnum[0,12] \colnum[1,9] \colnum[2,8] \colnum[3,6] \colnum[4,5]
    \colnum[5,3] \rownum[0,1] \rownum[1,2] \rownum[2,4] \rownum[3,7]
    \rownum[4,10] \rownum[5,11] \rownum[6,13]
\psline[linecolor=red,arrowsize=.4, arrowlength=.6]{->}(0,-.5)(2.5,-.5)(2.5,-1.5)(3.5,-1.5)(3.5,-2.5)(5.5,-2.5)
\end{pspicture}\hfill
 \begin{pspicture}(-1,1)(6,-7) 
\psline{C-C}(0,-7)(0,0)(6,0)
\cell(1,1)[] \cell(1,2)[] \cell(1,3)[\DT] \cell(1,4)[] \cell(1,5)[] \cell(1,6)[\DT] \cell(2,1)[] \cell(2,2)[] \cell(2,3)[] \cell(2,4)[\DT] \cell(2,5)[\DT] \cell(2,6)[] \cell(3,1)[] \cell(3,2)[] \cell(3,3)[] \cell(3,4)[] \cell(3,5)[\DT] \cell(4,1)[] \cell(4,2)[\DT] \cell(4,3)[] \cell(5,1)[\DT] \cell(6,1)[\DT] 
   \colnum[0,12] \colnum[1,9] \colnum[2,8] \colnum[3,6] \colnum[4,5]
    \colnum[5,3] \rownum[0,1] \rownum[1,2] \rownum[2,4] \rownum[3,7]
    \rownum[4,10] \rownum[5,11] \rownum[6,13]
\psline[linecolor=red,arrowsize=.4, arrowlength=.6]{->}(0,-.5)(2.5,-.5)(2.5,-4.5)
\end{pspicture}
\caption{Zigzag paths on a permutation tableau $T$ (left), on the alternative
  representation of $T$ (middle) and on the bare representation of $T$
  (right).}\label{fig:threetab}
\end{figure}
 
Steingr\'imsson and Williams \cite{Steingrimsson2007} showed that the map
$\zp:\PT(n)\to S_n$ is in fact a bijection preserving many interesting
statistics.  Similarly, we can define the zigzag map $\za:\PT(n)\to S_n$ on the
alternative representation in the same way: whenever we meet an arrow we change
the direction, see Figure~\ref{fig:threetab}.

In this section we show that $\za$ is essentially the same as the first
bijection $\CN:\PT(n)\to S_n$ of Corteel and Nadeau \cite{Corteel2009}.

Given $T\in\PT(n)$, we can consider $\za(T)$ as a product of cycles as follows.
Assume that $T$ has $k$ columns with labels $c_1<c_2<\cdots<c_k$. Note that each
column has exactly one $\uparrow$ and may have several $\leftarrow$'s. For each
Column $c_i$ let $C_i$ be the cycle
$(r^{(i)}_1,r^{(i)}_2,\dots,r^{(i)}_{\ell_{i}},c_i,r^{(i)}_0)$ where $r^{(i)}_0$
is the label of the row containing the unique $\uparrow$ in Column $c_i$ and
$r^{(i)}_1<r^{(i)}_2<\cdots<r^{(i)}_{\ell_{i}}$ are the labels of the rows
containing a $\leftarrow$ in Column $c_i$. It is easy to see that $\za(T)$ is
equal to the cycle product $C_1C_2\cdots C_k$. For example, if $T$ is the
permutation tableau in Figure~\ref{fig:threetab}, then
\[
\za(T) = (3,1) (5,2) (4,6,2) ( 2,8,1) (9,7) (12,10).
\]

Let $\pi^{(i)}=C_iC_{i+1}\cdots C_k$ for $i\in[k+1]$. Thus $\pi^{(1)}$ is equal
to $\za(T)$ and $\pi^{(k+1)}$ is the identity. We represent $\pi^{(i)}$ as a
product of cycles as follows. First we represent $\pi^{(k+1)} =
(m_1)(m_2)\cdots(m_s)$ where $m_1<m_2<\cdots<m_s$ are the labels of the
unrestricted rows of $T$. For each $i\in[k]$, we represent $\pi^{(i)}$ using the
representation of $\pi^{(i+1)}$ as follows. From the construction it is easy to
see that among the elements in the cycle
$C_i=(r^{(i)}_1,r^{(i)}_2,\dots,r^{(i)}_{\ell_{i}},c_i,r^{(i)}_0)$, $r^{(i)}_0$
is the only integer which appears in the cycle product representation of
$\pi^{(i+1)}$. Note that for different integers
$a_1,\dots,a_p,b_1,\dots,b_q,d_1,\dots,d_r$ and $x$, we have
\[(a_1,\dots,a_p,x)(b_1,\dots,b_q,x, d_1,\dots,d_r) =
(b_1,\dots,b_q,a_1,\dots,a_p,x,d_1,\dots,d_r).\] 

By this product rule, we represent $\pi^{(i)}$ as the one obtained by inserting
$r^{(i)}_1,r^{(i)}_2,\dots,r^{(i)}_{\ell_{i}},c_i$ before $r^{(i)}_0$ in the
representation of $\pi^{(i+1)}$.  Since this insertion process is exactly the
same as that in the definition of the map $\CN$, we obtain that $\pi^{(1)}$ is
equal to $\varphi\circ\CN(T)$. Here, for a permutation $\sigma$,
$\varphi(\sigma)$ denotes the permutation consisting of the cycles of the word
$\sigma$ according to RL-minima. Thus we get the following theorem. 

\begin{thm}\label{thm:10}
  The zigzag map $\za$ is the same as $\varphi\circ\CN$.
\end{thm}

For example, if $T$ is the permutation tableau in Figure~\ref{fig:threetab}, we
have
  \[\CN(T)=4,6,5,2,8,3,1,9,7,12,10,11,13,\]
\[\varphi\circ\CN(T)=(4,6,5,2,8,3,1)(9,7)(12,10)(11)(13)=\za(T).\]

\begin{remark}
  Burstein \cite{Burstein2007} defined the \emph{bare representation} of
  $T\in\PT(n)$ to be the diagram obtained from $T$ by replacing the topmost or
  the leftmost $1$'s with dots and removing the rest of the $0$'s and $1$'s, see
  Figure~\ref{fig:threetab}. In the same paper, he defined a map, say $\Theta$,
  on the bare representation by considering it as a binary forest. Similarly,
  one can check that $\varphi\circ \Theta$ is in fact the zigzag map on the bare
  representation.  As Nadeau mentioned in \cite{Nadeau}, though not explicitly
  stated, one can obtain a bare representation from an alternative
  representation (of another permutation tableau) very easily: replace all the
  arrows with dots. Since there is exactly one $\uparrow$ in each column and it
  is the topmost arrow, we can reconstruct the alternative representation.  It
  is worth noting that although the three bijections $\zp$, $\Phi$ and $\Theta$
  have very different descriptions, they are all zigzag maps on different
  representations.
\end{remark}

We can do the same thing for type B permutation tableaux. To state more
precisely, we define the following.  A \emph{zigzag path} on the alternative
representation of $T\in\PTB(n)$ is defined in the same way with the obvious
additional change of direction when we encounter the diagonal line, see
Figure~\ref{fig:typeB}. Note that a zigzag path for $T\in\PTB(n)$ always starts
from a row.

\begin{figure}
  \centering
\begin{pspicture}(-1,1)(6,-11) 
\cell(2,1)[]
  \rput(0,-1){\multido{\n=0+1}{6}{\rput(\n,-\n){\psline(1,0)}}}
\psline{c-c}(6,-6)(0,0)(0,-11)
 \cell(3,1)[\UP] \cell(3,2)[] \cell(4,1)[] \cell(4,2)[]
  \cell(4,3)[] \cell(5,1)[] \cell(5,2)[] \cell(5,3)[\LT]
  \cell(5,4)[\UP] \cell(6,1)[] \cell(6,2)[] \cell(6,3)[]
  \cell(6,4)[] \cell(6,5)[] \cell(7,1)[] \cell(7,2)[] \cell(7,3)[]
  \cell(7,4)[] \cell(7,5)[\LT] \cell(7,6)[\UP] \cell(8,1)[] \cell(8,2)[]
  \cell(8,3)[] \cell(8,4)[] \cell(9,1)[] \cell(9,2)[] \cell(9,3)[\LT]
  \cell(9,4)[] \cell(10,1)[\LT] \cell(10,2)[\UP] \cell(10,3)[]
  \rownum[0,-10] \rownum[1,-9] \rownum[2,-8] \rownum[3,-6]
  \rownum[4,-3] \rownum[5,-2] \rownum[6,1] \rownum[7,4] \rownum[8,5]
  \rownum[9,7] \rownum[10,11] 
    \colnum[0,10] \colnum[1,9] \colnum[2,8] \colnum[3,6] \colnum[4,3] \colnum[5,2] 
  \psline[linecolor=red,arrowsize=.4, arrowlength=.6]{->}(0,-3.5)(3.5,-3.5)(3.5,-4.5)(4.5,-4.5)(4.5,-6.5)(5.5,-6.5)(5.5,-7.5)
\end{pspicture}
\caption{A zigzag map on the alternative representation of some $T\in\PT(n)$.}
  \label{fig:typeB}
\end{figure}
 
For $T\in\PTB(n)$ and $i\in[n]$, we define $i_T=i$ if Column $i$ has a
$\uparrow$, and $i_T=-i$ otherwise. Let $R(T)=\{1_T,2_T,\dots,n_T\}$. Note that
$R(T)$ is the set of integers appearing in the permutation $\CNB(T)$.  Then we
define $\zaB(T)$ to be the map $\pi:R(T)\to R(T)$ such that for each $i\in[n]$,
if the zigzag path entering from Row $i$ or Row $-i$ exits $T$ from Row $j$ or
Column $j$, then $\pi(i_T)=j_T$. It is easy to see that $\pi: R(T) \to R(T)$ is
a bijection.

\begin{example}
  For $T$ in Figure~\ref{fig:typeB}, we have
  $R(T)=\{1,2,-3,4,5,6,7,-8,9,10,11\}$ and $\zaB(T)= \pi:R(T)\to R(T)$ is the
  following
  \begin{equation}
    \label{eq:4}
\pi = \left(\begin{array}{rrrrrrrrrrr}
  1 & 2 & - 3 &4& 5& 6& 7& -8& 9& 10& 11\\
  -3 & 1 & 5 & 4 & -8 & 2 & 10 & 9 & 7 & 6 & 11
\end{array} \right),
 \end{equation}
where the bi-letter $\begin{array}{cc} i\\ j \end{array}$ means that $\pi(i) =j$.
\end{example}

Note that we can identify $\zaB(T)=\pi$ with an element in $B_n$ by reading the
lower line in the two line notation of $\pi$ as shown in \eqref{eq:4}. For
example, the $\pi$ in \eqref{eq:4} is identified with $ -3 , 1 , 5 , 4 , -8 , 2
, 10 , 9 , 7 , 6 , 11 \in B_{11}$. Thus $\zaB$ is a map from $\PTB(n)$ to
$B_n$.

We extend the definition of the map $\varphi$ on $S_n$ to $B_n$ as follows. For
$\pi\in B_n$, we define $\varphi(\pi)$ to be the permutation on the set of
integers in $\pi$ consisting of the cycles of the word $\pi$ according to
RL-minima.  For example, if $\sigma=9,7,10,6,2,1,-3,5,-8,4,11$, then
\[
\varphi(\sigma)=(9,7,10,6,2,1,-3,5,-8)(4)(11).
\]

By the same argument used for proving Theorem~\ref{thm:10}, we obtain the
following theorem.

\begin{thm}\label{thm:11}
  The zigzag map $\zaB:\PTB(n)\to B_n$ is the same as $\varphi\circ \CNB$.
\end{thm}

For example, if $T$ is the tableau in Figure~\ref{fig:typeB}, then
\[\CNB(T)=9,7,10,6,2,1,-3,5,-8,4,11,\]
\[ \zaB(T)=\varphi\circ \CNB(T)=(9,7,10,6,2,1,-3,5,-8)(4)(11).\]

\section{Further study}
\label{sec:further-study}

In Corollary~\ref{thm:5} we have found the following sign-imbalance formula for
permutation tableaux: if $n=4k+r$ for $0\leq r<4$, then
\begin{equation}
  \label{eq:14}
\sum_{T\in\PT(n)}\sgn(T) =\frac{(1+i)^n + (1-i)^n}{2}= \left\{
  \begin{array}{ll}
(-1)^k \cdot 2^{2k}, & \mbox{if $r=0$ or $r=1$,}\\
0, & \mbox{if $r=2$,}\\
(-1)^{k+1} \cdot 2^{2k+1}, & \mbox{if $r=3$.}
 \end{array} \right.
\end{equation}
Since we have obtained \eqref{eq:14} by putting $t=-1$ in the generating
function in Theorem~\ref{thm:ptx}, we propose the following problem.

\begin{problem}
  Find a combinatorial proof of \eqref{eq:14}.
\end{problem}

There is also a sign-imbalance formula for standard Young tableaux. The sign of
a standard Young tableau $T$ is defined to be the sign of the permutation
obtained by reading $T$ like a book. For example,
\[
\sgn\left(
\raisebox{-.4cm}{\begin{pspicture}(-.1,0)(3,-2)
\cell(1,1)[1] \cell(1,2)[2] \cell(1,3)[5]
\cell(2,1)[3] \cell(2,2)[4]    
\psline(0,-2)(0,0)(3,0)
\end{pspicture} }
\right) = \sgn(1 2 5 34 ) =1.
\]

Stanley \cite{Stanley2005} conjectured the following sign-imbalance formula for
standard Young tableaux:
\begin{equation}
  \label{eq:1}
\sum_{T\in \mathcal{SYT}(n)}\sgn(T)=2^{\left\lfloor \frac n2 \right\rfloor},
\end{equation}
where $\mathcal{SYT}(n)$ is the set of standard Young tableaux with $n$ squares.
Lam \cite{Lam2004} and Sj\"ostrand \cite{Sjostrand2005} independently proved
\eqref{eq:1}, and Kim \cite[Corollary 4.10 and Theorem 4.13]{JSK_imbalance} generalized \eqref{eq:1} to skew
standard Young tableaux.

If we take the absolute values of \eqref{eq:14} and \eqref{eq:1}, we obtain the
following somewhat unexpected result: if $n\not\equiv 2 \mod 4$, then
\begin{equation}
  \label{eq:5}
\left| \sum_{T\in \PT(n)}\sgn(T) \right| =
\left| \sum_{T\in \mathcal{SYT}(n)}\sgn(T) \right|
=2^{\left\lfloor \frac n2 \right\rfloor}.
\end{equation}
In \eqref{eq:5} taking the absolute value of the sign-imbalance for standard
Young tableaux is only for an aesthetic reason. It will be interesting to find a
combinatorial explanation for \eqref{eq:5}.

Since we do not have a type B analog for \eqref{eq:14}, we propose the
following problem.

\begin{problem}
  Define the sign of a type B permutation tableau and find a sign-imbalance for
  type B permutation tableaux.
\end{problem}

\section*{Acknowledgement}
The authors would like to thank Matthieu Josuat-Verg\`es for helpful discussion.
The authors also thank the anonymous referees for their careful reading and
helpful comments, especially the comment for simplifying the proof of
Theorem~\ref{thm:ptx}.

\providecommand{\bysame}{\leavevmode\hbox to3em{\hrulefill}\thinspace}
\providecommand{\MR}{\relax\ifhmode\unskip\space\fi MR }
\providecommand{\MRhref}[2]{%
  \href{http://www.ams.org/mathscinet-getitem?mr=#1}{#2}
}
\providecommand{\href}[2]{#2}

\end{document}